\documentclass[11pt,reqno]{amsart}

\usepackage[lite,backrefs]{amsrefs} 
\usepackage[bookmarks, colorlinks=true, linkcolor=blue, citecolor=blue, urlcolor=blue, pagebackref=true]{hyperref}

\usepackage{mathtools}
\usepackage{amssymb}
\usepackage{calligra,mathrsfs} 
\usepackage{hyperref}
\usepackage{imakeidx}
\usepackage{tikz}
\usepackage{manfnt}
\usetikzlibrary{%
  matrix,%
  calc,%
  arrows%
}
\usepackage{graphicx}
\usepackage{fullpage}

\newcommand{\Z}{\mathbb{Z}}
\newcommand{\comment}[1]{}

\DeclareMathOperator{\Ass}{Ass}
\DeclareMathOperator{\Min}{Min}

\def\1{{\bf 1}}
\def\0{{\bf 0}}

\numberwithin{equation}{section}
\theoremstyle{plain}
\newtheorem{theorem}{Theorem}[section]

\newtheorem{corollary}[theorem]{Corollary}

\newtheorem{lemma}[theorem]{Lemma}

\theoremstyle{definition}
\newtheorem{definition}[theorem]{Definition}

\newtheorem{remark}[theorem]{Remark}
\newtheorem{example}[theorem]{Example}

\newtheorem{question}[theorem]{Question}

\makeatletter
\def\imod#1{\allowbreak\mkern10mu\left({\operator@font mod}\,\,#1\right)}
\makeatother

\begin{document}

\bibliographystyle{plain}

\title[SwanWalkarXiv]{Tensor-Multinomial Sums of Ideals: Primary Decompositions and Persistence of Associated Primes}
\author{Irena Swanson and Robert M. Walker}

\address{Department of Mathematics, Reed College, 3203 SE Woodstock Blvd, Portland, OR, 97202}
\email{iswanson@reed.edu}

\address{Department of Mathematics, University of Michigan, Ann Arbor, MI, 48109}
\email{robmarsw@umich.edu}

\begin{abstract} 
Given a polynomial ring $C$ over a field
and proper ideals $I$ and $J$
whose generating sets involve disjoint variables,
we determine how to embed the associated primes of each power of $I+J$ into a collection of primes described in terms of the associated primes of select powers of $I$ and of $J$.
We record two applications. First, in case the field is algebraically closed,
we construct primary decompositions for powers of $I+J$ from primary decompositions for powers of $I$ and $J$. Separately, we attack the persistence problem for associated primes of powers of an ideal in case one of $I$ or $J$ is a non-zero normal ideal. 
\end{abstract}


\maketitle

\section{Introduction}

Throughout, $A = K[x_1, \ldots, x_a]$,  $B = K[y_1, \ldots, y_b]$, and $C = A \otimes_K B = K[x_1, \ldots , x_a, y_1, \ldots, y_b]$, where $x_1, \ldots, x_a, y_1, \ldots, y_b$ are variables over a field $K$.  We fix ideals $I \subsetneqq A$ and $J \subsetneqq B$. By abuse of notation, we use the same symbol to denote ideals in $A$ or $B$ and their expansions to $C$. 

Our primary focus is on constructing  primary decompositions 
for the powers of the ideal $$I+J \subsetneqq C,$$
which we call a \textbf{(two-term) tensor-multinomial sum of ideals}, clarifying our chosen title.

By a result of Brodmann \cite{Brodmann} (or \cite[Theorem~3.5]{5authorSymbolicSurvey}),
the collection $\mathcal{A}_A (I) := \bigcup_{n=1}^\infty \operatorname{Ass}_A (A/I^n)$
of associated primes of powers of $I$ is a finite set.
By our abuse of notation, $\mathcal{A}_A (I) = \mathcal{A}_C (I)$.

\begin{definition}\label{def:powersprimarydecomp}
Suppose $\mathcal{A}_A (I) = \mathcal{A}_C (I) =\{P_1, \ldots, P_r\}$ 
and $\mathcal{A}_B (J) = \mathcal{A}_C (J) =\{Q_1, \ldots, Q_s\}$.  For each integer $m \ge 1$, fix primary decompositions   
$$
I^m = p_{m1} \cap \cdots \cap p_{mr} \quad \mbox{ and } \quad  J^m = q_{m1} \cap \cdots \cap q_{ms},
$$
where for each $k = 1, \ldots, r$, the ideal 
$p_{mk}$ is either $P_k$-primary or equals $C$,
and for $\ell = 1, \ldots, s$, the ideal 
$q_{m\ell}$ is either $Q_\ell$-primary or else $C$.
We also set $p_{0k} = q_{0\ell} = C$ for all $k, \ell$.
\end{definition}

\begin{definition}\label{def:SwanWalkfiltration}
A list of ideals $\{L_c\}_{c \ge 0}$ in $C$ is a {\bf filtration} if $L_0 = C$ and $L_c \supseteq L_{c+1}$ for all $c \ge 0$. 
\end{definition}
\begin{remark} This definition differs from H\`{a}--Nguyen--Trung--Trung~\cite[Section 3]{HNTTBinomial} only in that we \textit{do not} require $L_1$ to be non-zero and proper -- upon inspection this assumption is inessential for the proofs of \cite[Proposition~3.3, Theorem~3.4]{HNTTBinomial}. We briefly recall this point at the start of Section~\ref{sect:main}. \end{remark}

The technical heart of the paper lies in the following theorem, deduced in Section \ref{sect:main}: 
\begin{theorem}\label{thm:SwanWalkC}
Fix a field $K$, and polynomial $K$-algebras 
$A = K[x_1, \ldots, x_a]$,  
$B = K[y_1, \ldots, y_b]$
and $C = A \otimes_K B = K[x_1, \ldots , x_a, y_1, \ldots, y_b]$.
Let $I$ be an ideal in $A$ and $J$ an ideal in $B$.
Then for each integer $n >0$,
$$
\Ass(C/(I+J)^n)
\subseteq
\bigcup \left\{\mathcal{P} \in \Min(C/(P+Q)) :
P \in \bigcup_{i=1}^n \Ass(C/I^i), Q \in \bigcup_{j=1}^n \Ass(C/J^{j})
\right\},
$$ 
and moreover
$$
\mathcal{A}_C(I+J) =
\bigcup \{\mathcal{P} \in \Min(C/(P+Q)) : P \in
\mathcal{A}(I), Q \in \mathcal{A}(J)\}.
$$
\end{theorem}

We record two applications of Theorem \ref{thm:SwanWalkC}. First, we record the following theorem, our primary result in the paper -- we construct primary decompositions for powers of $I + J$
from filtered primary components for powers of $I$ and $J$ 
in the case where $K$ is algebraically closed.

\begin{theorem}[Cf. Corollary  \ref{thm:SwanWalkD}]\label{thm:SwanWalkA}
Let $K$ be an algebraically closed field and suppose that for each pair $1 \le k \le r$ and $1 \le \ell \le s$, the collections 
$\{p_{mk}\}_{m=0}^\infty$ and $\{q_{m\ell}\}_{m=0}^\infty$ from Definition \ref{def:powersprimarydecomp} are filtrations.
Then for each positive integer $n$, 
$$
(I+J)^n =
\bigcap_{k = 1}^r
\bigcap_{\ell = 1}^s
\left(
\sum_{i=0}^n
p_{ik} \cdot q_{n-i, \ell} 
\right),
$$
where each $\sum_{i=0}^n
p_{ik} \cdot q_{n-i, \ell}$
is either $C$ or primary to the prime ideal $P_k + Q_\ell$.
Furthermore,
$$
\mathcal{A}_C(I+J) =  \left\{P_k + Q_\ell \colon 1 \le k \le r , 1 \le \ell \le s \right\},$$ 
and so the cardinalities of these finite sets satisfy the relation $$\#\mathcal{A}_C(I+J) = \#\mathcal{A}_A(I) \cdot \#\mathcal{A}_B(J) .$$ 
\end{theorem}
\noindent Lemma~\ref{lem:monotone} below indicates how to easily construct filtered collections -- such as $\{p_{mk}\}_{m=0}^\infty$ and $\{q_{m\ell}\}_{m=0}^\infty$ in Theorem \ref{thm:SwanWalkA} -- from arbitrary primary decompositions for powers of $I$, $J$, and $I+J$.

Prior to obtaining Theorem \ref{thm:SwanWalkC}, we  originally pursued Theorem \ref{thm:SwanWalkA} with a view towards attacking the following question on \textit{persistence} of associated primes. 

\begin{question}\label{ques:SwanWalk001}
Suppose that 
$\Ass(A/I^{n-1}) \subseteq \Ass(A/I^n)$ and  $\Ass(B/J^{n-1}) \subseteq \Ass(B/J^n)$ for all integers $n>0$. 
When is it the case that  
$\Ass(C/(I+J)^{n-1}) \subseteq \Ass(C/(I+J)^n)$ for all $n$ as well?
\end{question}
\noindent This persistence property seems to hold for prime ideals in polynomial rings that we have been able to find in the literature and study using Macaulay2 \cite{M2}. We know of no prime ideal that fails to satisfy persistence. That said, it seems unlikely that this persistence of associated primes holds for all prime ideals in any polynomial ring over an  arbitrary ground field.

As a second application of Theorem \ref{thm:SwanWalkC}, we  answer Question \ref{ques:SwanWalk001} affirmatively when one of $I$ or $J$ is a nonzero \textbf{normal ideal}, i.e., when all the powers of $I$ or of $J$ are integrally closed. Said ideal satisfies the persistence property by a result of Katz and Ratliff~\cite[(1.3) Theorem]{KR},  and even if the other ideal does not, the corollary to follow indicates that their sum does satisfy it, indicating that the persistence property is \textit{remarkably} persistent and robust under extension of scalars. 

\begin{corollary}\label{thm:SwanWalkB}
\textit{
Fix a field $K$, along with polynomial rings $A = K[x_1, \ldots, x_a]$,  
$B = K[y_1, \ldots, y_b]$
and $C = A \otimes_K B = K[x_1, \ldots , x_a, y_1, \ldots, y_b]$.
Let $I$ be an ideal in $A$ and $J$ a non-zero normal ideal in $B$.
Then 
$\Ass(C/(I+J)^{n-1}) \subseteq \Ass(C/(I+J)^n)$ for all positive integers $n$.
}
\end{corollary}

\begin{remark}\label{rem:SwanWalkHNTT}  Theorems \ref{thm:SwanWalkC} and \ref{thm:SwanWalkA} and Corollary \ref{thm:SwanWalkB} extend in two directions. Our proofs extend to the setting considered by H\`{a}--Nguyen--Trung--Trung~\cite{HNTTBinomial}. Namely, they fix Noetherian commutative algebras $A$ and $B$ over a common field $K$, such that $C = A \otimes_K B$ is Noetherian as well, along with non-zero ideals $I \subseteq A$ and $J \subseteq B$. Moreover, both results  can be rendered for finite tensor products and tensor-multinomial sums of ideals, and deduced via inductive arguments; see the second author's paper \cite[Proof of Multinomial Theorem 2.8]{Walker003} which is instructive in this vein. \end{remark}

\noindent \textbf{Acknowledgements:} We thank Karen E. Smith for encouraging our collaboration and for editing several drafts to improve exposition. The second author acknowledges support from NSF RTG grant DMS-0943832, a 2017-18 Ford Foundation Dissertation Fellowship, and a 2018-19 Rackham Science Award from the Rackham Graduate School at UM-Ann Arbor. Several instructive computations were performed using Macaulay2 \cite{M2}, as reflected by the examples recorded in Section \ref{sect:examples}.

\section{Preparatory Lemmas }

We continue to work with the polynomial $K$-algebras  $A$, $B$, and $C$ as in the Introduction. Of the lemmas  recorded here, the ones cited in proofs in Section \ref{sect:main} are Lemmas \ref{lem:monotone},  \ref{lem:main002}, and \ref{lem:colon} -- \ref{lem:colon3}.  

\begin{lemma}\label{lem:monotone} Consider the ideal collection $\{p_{mk}\}_{m=0}^\infty$ from Definition~\ref{def:powersprimarydecomp} for a fixed index $1 \le k \le r$. If we set $p'_{mk} = \cap_{i = 0}^m p_{ik}$, then $\{p'_{mk}\}_{m=0}^\infty$ is a filtration such that $p'_{mk}$ is either $C$ or $P_k$-primary for each $m \ge 0$, and $I^m = \cap_k p'_{mk}$ is a primary decomposition with possible redundancies. \end{lemma}

\begin{proof}
Each $p'_{mk}$ is a finite intersection of ideals that are either $C$ or $P_k$-primary, and hence $p'_{mk}$ is either $C$ of $P_k$-primary.
Since $I^i \supseteq I^m$ for all $i \le m$, it follows that any $P_k$-primary component of $I^i$ contains $I^m$. Thus $I^m \subseteq \cap_k \cap_{i=0}^m p_{ik} = \cap_k p'_{mk} \subseteq \cap_k p_{mk} = I^m$, so equality holds throughout.
\end{proof}

\begin{lemma}\label{lem:idealoperatorID} For any ideals $I \subseteq A$ and $J_1, J_2 \subseteq B$, $$(I + J_1) \cap J_2 = IJ_2 + J_1 \cap J_2.$$ Similarly, for any ideals $I_1, I_2 \subseteq A$ and $J \subseteq B$, $I_1 \cap (I_2 + J) = I_1 \cap I_2 + I_1 J$. In particular, $I \cap J = I J$. \end{lemma}

\begin{proof}
We only prove the displayed equality
because the second statement follows by symmetry,
and because the last statement follows trivially from it.
We adapt the proof for  \cite[Lemma~3.1]{HNTTBinomial}.
First, some notation: given sets $U$ and $V$ in $A$ and $B$, respectively, their simple tensor set is 
$$U \otimes V =\{u \otimes v \colon u \in U, v \in V\}.$$ 
Let $U$ be a $K$-vector space basis for $I$, and $V$ a $K$-vector space basis for $J_1 \cap J_2$. 
Extend $V$ to a $K$-basis $V_i$ for each $J_i$,
and extend $U$ to a $K$-basis $U^*$ for $A$ and $V_2$ to a $K$-basis $V_2^*$ for $B$.
Then $U^*\otimes V_2^*$ is a $K$-basis for $A \otimes_K B$. 
Notice $(I + J_1) \cap J_2$ is generated by 
$$\bigl(U \otimes  V_2^*) \cup (U^* \otimes V_1) \bigr)
\cap (U^* \otimes V_2) = (U \otimes V_2) \cup \bigl( U^* \otimes (V_1 \cap V_2) \bigr)$$
and the right-hand side generates $I J_2 + (J_1 \cap J_2)$. 
\end{proof}

\begin{lemma}\label{lem:main001}
For any index $k \in \Z_{> 0}$, let $\mathcal{I}_k = \{I_{ik}\}_{i=0}^\infty$ consist of ideals in $A$
and $\mathcal{J} = \{J_{i}\}_{i=1}^\infty$ consist of ideals in $B$ with $J_0 \supseteq J_1 \supseteq J_2 \supseteq \cdots$.
Then for any pair of integers $r \ge 1$ and $n \ge 0$, 
$$
\bigcap_{k=1}^r \left(\sum_{i=0}^n I_{ik} J_{n-i}\right) =
\sum_{i=0}^n \left(\bigcap_{k=1}^r \widetilde I_{ik}\right) J_{n-i},
$$ where $\widetilde I_{ik} = \sum_{j = i}^n I_{jk}$.
When $\mathcal{I}_k$ is a filtration in $C$ for a given index $k$,
then in fact $\widetilde I_{ik} = I_{ik}$.

An analogous identity holds when the roles of $A$ and $B$ are switched. 
\end{lemma}

\begin{proof}
We induce on $r$ and then on $n$. The case $r=1$ is trivial. Observe that for all $k$,
$$
\sum_{i=0}^n I_{ik} J_{n-i}
= \sum_{i=0}^n \widetilde I_{ik} J_{n-i}.
$$
Replacing the $I_{ik}$ with $\tilde I_{ik}$ for all $k$, 
it suffices to prove the lemma assuming the
$\mathcal{I}_k$ are filtrations.

By induction it suffices to prove the case $r = 2$.
Several times in the proof we will use Lemma~\ref{lem:idealoperatorID}. Since $I_{0 k} = C$ for all $k$, the claim holds for $n = 0$ as $
I_{01}J_0 \cap I_{02}J_0  = J_0
= (I_{01} \cap I_{02}) J_0$. 
Now assume that $n > 0$, assuming the identity for $r=2$ and $n-1$.
The first three equalities below, along with the sixth, 
use the easy fact that for any ideals $L_1, L_2, L_3$ in a ring $R$,
if $L_1 \subseteq L_3$
then $(L_1 + L_2) \cap L_3 = L_1 + (L_2 \cap L_3)$; if $n \ge 2$, the fifth holds by applying Lemma \ref{lem:idealoperatorID} to the first boxed intersection and applying the induction hypothesis to the latter, first replacing $I_{ik}$ with $I'_{ik} := I_{i+1,k}$ for $1 \le i \le n-1$. Set $\mathcal{S}_{k, t, u} := \sum_{i=u}^{t} I_{ik} J_{t-i}$ for pairs $t \ge 1$ and $u \ge 1$. Note that $\mathcal{S}_{k, t, u} = C$ if $t <u $ and otherwise lies in $I_{vk}$ for triples $t \ge u \ge v \ge 0$.  The left-hand intersection is 
\begin{align*}
&
\left(J_n + \mathcal{S}_{1,n,1} \right)
\cap \left(J_n + \mathcal{S}_{2,n,1} \right)  = J_n + ((J_n + \mathcal{S}_{1,n,1}) \cap \mathcal{S}_{2,n,1}) \quad \quad \quad  [L_1 = J_n \subseteq L_3 = J_n + \mathcal{S}_{1,n,1}]\\\
&=
J_n + 
\biggl(\biggl(J_n + (I_{11} \cap \left(J_{n-1} + \mathcal{S}_{1, n, 2}) \right)\biggr)  \cap \left(I_{12} \cap (J_{n-1} + \mathcal{S}_{2, n, 2} ) \right) \biggr) \quad [L_1  = \mathcal{S}_{k,n,1} \subseteq L_3 = I_{1 k}, \mbox{ } k=1,2] \\
&=
J_n + 
\biggl((J_n + I_{11}) \cap \left(J_{n-1} + \mathcal{S}_{1, n, 2} \right)  \cap \left(I_{12} \cap (J_{n-1} + \mathcal{S}_{2, n, 2} ) \right) \biggr) \quad [L_1 = J_n \subseteq L_3 = J_{n-1} + \mathcal{S}_{1, n, 2}] \\
&=
J_n + \biggl(
{(J_n + I_{11}) \cap I_{12}}
\cap {\left(J_{n-1} + \mathcal{S}_{1, n, 2}\right)
\cap \left(  J_{n-1} + \mathcal{S}_{2, n, 2} \right)} \biggr)  \quad   [\mbox{by reordering}]. \\
\end{align*}
If $n=1$, we are now done with Lemma \ref{lem:idealoperatorID},
otherwise we continue with equalities:
\begin{align*}
&= J_n + \biggl( 
(J_n I_{12} + I_{11} \cap I_{12})
\cap \left(J_{n-1} + \sum_{i=2}^n (I_{i1} \cap I_{i2})J_{n-i} \right) \biggr) 
\\
&= J_n + \biggl( 
((J_n I_{12} + I_{11} \cap I_{12})
\cap J_{n-1}) + \sum_{i=2}^n (I_{i1} \cap I_{i2})J_{n-i} \biggr) \quad  [L_1 = \sum_{i=2}^n (I_{i1} \cap I_{i2})J_{n-i} \subseteq L_3 = I_{11} \cap I_{12}]
\\
&= J_n + \left( 
J_n I_{12} + \left( \sum_{i=1}^n (I_{i1} \cap I_{i2})J_{n-i}\right) \right) \quad \quad \quad  [L_1 = J_n I_{12} \subseteq L_3 = J_{n-1} ],\end{align*}
which certainly equals the desired right-hand sum. 
The lemma then follows in full.
\end{proof}

\begin{remark} The same argument proves that if the $J_i$ form a chain of ideals,
then
$$
\bigcap_{k=1}^r \left(\sum_{i=0}^n I_{ik} J_{n-i}\right) =
\sum_{i=0}^n \left(\bigcap_{k=1}^r \widehat I_{ik}\right) J_{n-i},
$$
where $\widehat I_{ik}$ is the sum of those $I_{jk}$
for which $J_{n-i} \subseteq J_{n-j}$. Indeed, this follows 
from the identity 
$I_{ik} J_{n-i} + I_{jk} J_{n-j} =
I_{ik} J_{n-i}
+ I_{jk} J_{n-i} 
+ I_{jk} J_{n-j}$.
\end{remark}

\begin{lemma}\label{lem:main002}
Suppose the ideal collections $\{p_{nk}\}_{n=0}^\infty$ and $\{q_{n\ell}\}_{n=0}^\infty$ from Definition~\ref{def:powersprimarydecomp}
are filtrations for each fixed pair $1 \le k \le r$ and $1 \le \ell \le s$.
Then for each positive integer $n$, 
$$
(I+J)^n =
\bigcap_{\ell = 1}^s
\bigcap_{k = 1}^r
\left(
\sum_{i=0}^n
p_{ik} \cdot q_{n-i,\ell} 
\right).
$$
\end{lemma}

\proof
We invoke Lemma \ref{lem:main001} twice:
\begin{align*}
\bigcap_{\ell = 1}^s \bigcap_{k = 1}^r
\left( \sum_{i=0}^n p_{ik} q_{n-i,\ell} \right)
=
\bigcap_\ell \left(
\sum_{i=0}^n \bigl(\bigcap_k p_{ik}\bigr) q_{n-i,\ell} \right) 
=
\bigcap_\ell \left(
\sum_{i=0}^n I^i q_{n-i,\ell} \right) 
&=
\sum_{i=0}^n I^i \bigl(\bigcap_\ell q_{n-i,\ell}\bigr) \\
& = \sum_{i=0}^n I^i J^{n-i} = (I+J)^n. & \qed \\
\end{align*}

\begin{remark} If $P_k$ is not associated to $I, I^2, \ldots, I^n$, then no components involving the $p_{ik}$ are needed in the decomposition of $(I+J)^n$ in the lemma above. However, if $P_k$ is associated to some $I^i$ for $i < n$, then $p_{1k}, \ldots, p_{nk}$ may or may not be needed, as shown in examples in Section~\ref{sect:examples}.  \end{remark}

\begin{lemma}\label{lem:colon}
Let $f$ be a non-zero divisor in $A$. 
Let $I$ be an ideal in $A$ and $J$ an ideal in $B$.
Then $(I+J) : f = (I : f) + J$.
\end{lemma}

\begin{proof}
Let $c \in (I + J) : f$.
Then by Lemma \ref{lem:idealoperatorID},
$$
cf \in (I + J) \cap (f) = I \cap (f) + J (f)
= (I : f) f + J (f),
$$
so that $c \in (I : f) + J$.
The other inclusion is easy.
\end{proof}

\begin{lemma}\label{lem:colon2}
Let $L_1 \subseteq L_2$ be proper ideals in a Noetherian ring $R$
and let $P$ be a prime ideal associated to $L_1$ but not to $L_2$.
Then $(L_1 : P) \cap L_2$ properly contains $L_1$
and there exists $f \in (L_1 : P) \cap L_2$
such that $L_1 : f = P$.
\end{lemma}

\begin{proof}
Certainly $L_1$ is contained in $(L_1 : P) \cap L_2$.
Suppose that $L_1 = (L_1 : P) \cap L_2$.
Let $p$ be a $P$-primary component of $L_1$.
Then $p : P$ is the only $P$-primary component on the right-hand side
of the ideal equality,
and so by the mix-and-match theorem of primary decompositions
due to Yao~\cite{yao2002primary}, 
$p : P$ is also a primary component of $L_1$
on the left-hand side.
But then $p : P^n$ is a primary component of $L_1$ for all non-negative
integers $n$,
but this is a contradiction as for large $n$,
$p : P^n = R$.
Thus $(L_1 : P) \cap L_2$ properly contains $L_1$.

Let $L'$ be the intersection of all primary components of $L_1$
whose radicals properly contain $L_1$.
Then $P$ is not associated to $L_2 \cap L'$
and $L_1 \subseteq L_2 \cap L'$.
Then by the previous paragraph
there exists $f \in (L_1 : P) \cap L_2 \cap L'$ such that $f \not \in L_1$.
Then $P \subseteq L_1 : f$,
and the latter is a proper ideal whose associated primes
are all associated to $L_1$
and none properly contain~$P$.
Thus $P = L_1 : f$.
\end{proof}

\begin{lemma}\label{lem:colon3}
Let $L$ be an ideal in a Noetherian ring $R$
containing a non-zerodivisor.
Let $n$ be a positive integer
and suppose that $L, L^2, \ldots, L^{n-1}$ are integrally closed.
Suppose that $P$ is associated to $L^n$.
Then there exists $f \in L^{n-1}$
such that $L^n : f = P$.
\end{lemma}

\begin{proof}
Let $q_n$ be the intersection of primary components of $L^n$
whose radicals properly contain $P$.
Certainly $L^n \subseteq (L^n : P) \cap L^{n-1} \cap q_n$.
Suppose that equality holds.
By Lemma~\ref{lem:colon2},
$P$ is associated to $L^{n-1} \cap q_n$, and hence to $L^{n-1}$.
Colon the equality by $L$:
$$
((L^n : L) : P) \cap (L^{n-1} : L) \cap (q_n : L) = L^n : L.
$$
By the determinantal trick due to Pr\"ufer
(see \cite[Corollary 1.1.8]{HunSwan})
for all positive integers $k$,
$L^k : L$ is a subset of the integral closure of $L^{k-1}$.
By assumption this equals $L^{k-1}$ if $k \le n$.
Thus
$$
(L^{n-1} : P) \cap L^{n-2} \cap (q_n : L) = L^{n-1}.
$$
By repeating this step
we get that $P$ is associated to $L^{n-2}, L^{n-3}, \ldots, L$
and that
$$
(L : P) \cap (q_n : L^{n-1}) = (L : P) \cap L^0 \cap (q_n : L^{n-1}) = L.
$$
If the $P$-primary component on the right is $p$,
then the only $P$-primary component on the left is $p : P$,
and so by the mix-and-match theorem of primary decompositions
due to Yao~\cite{yao2002primary},
in a primary decomposition of $L$
we can replace $p$ with $p : P$,
and similarly that with $p : P^2$,
et cetera.
But for large $m$,
$p : P^m = R$,
which says that $P$ is not associated to $L$ after all,
which is a contradiction.
Thus $L^n$ is properly contained in $(L^n : P) \cap L^{n-1} \cap q_n$.
Let $f$ be in the latter ideal and not in $L^n$.
Then $L^n : f$ is a proper ideal which contains $P$
and has no associated primes strictly larger than $P$,
so that $L^n : f = P$.
\end{proof}

\section{Proofs of the Key Results}\label{sect:main}

We continue to work with the polynomial $K$-algebras  $A$, $B$, and $C$ as in the Introduction. Of the lemmas  recorded above, we require Lemmas \ref{lem:monotone}, \ref{lem:main002} and  \ref{lem:colon} --\ref{lem:colon3}  going forward.  

Our proofs rely on H\`{a}--Nguyen--Trung--Trung~\cite[Lemma 2.4, Theorem 2.5, Proposition~3.3, Proof of Theorem~3.4]{HNTTBinomial}. 
The proofs of these results in~\cite{HNTTBinomial}
work for filtrations as defined in  Definition~\ref{def:SwanWalkfiltration}.
By ~\cite[Theorem~2.5]{HNTTBinomial}, given nonzero finitely-generated modules $M$ and $N$ over $A$ and $B$, respectively, 
\begin{equation}\label{eqn:SwanWalkKtensor-assoc-primes}
\Ass_C (M \otimes_K N) = \bigcup  \left\{P \in \operatorname{Min}_C (C /\mathfrak{p} + \mathfrak{q})  \colon \mathfrak{p} \in \Ass_A (M), \mbox{ }  \mathfrak{q} \in \Ass_B(N) \right\},
\end{equation}
in terms of sets of associated primes and minimal associated primes.  By~\cite[Proposition~3.3]{HNTTBinomial},
for any filtrations $\{I_i\}_{i \ge 0}$ and $\{J_j\}_{j \ge 0}$ in $A$ and $B$, respectively, we have for any integer $n \ge 0$ an isomorphism of $C$-modules deduced at the level of $K$-vector spaces: 
\begin{equation}\label{eqn:decomp}
\displaystyle
\frac{\sum_{i+j=n}I_iJ_j}{\sum_{i+j=n+1}I_iJ_j}
\cong \bigoplus_{i=0}^n\big(I_i/I_{i+1} \otimes_K J_{n-i}/J_{n-i+1}\big).
\end{equation}

\begin{lemma}\label{lem:main003}
Suppose that for each pair $1 \le k \le r$ and $1 \le \ell \le s$, the collections 
$\{p_{mk}\}_{m=0}^\infty$ and $\{q_{m\ell}\}_{m=0}^\infty$ from Definition \ref{def:powersprimarydecomp} are filtrations.
Then for each triple of integers $n \ge 1$, $k \in \{1, \ldots, r\}$, $\ell
\in \{1, \ldots, s\}$,
$$
\Ass_C \left(C/ (\sum_{i+j = n} p_{ik} \cdot q_{j \ell} ) \right) \subseteq \Min_C (C/(P_k + Q_\ell)).
$$ In particular, in case $\sum_{i+j = n} p_{ik} \cdot q_{j \ell}$ is proper and $P_k + Q_\ell$ is a prime ideal (e.g., this holds when $K$ is algebraically closed),  $\sum_{i+j = n} p_{ik} \cdot q_{j \ell}$ is primary to $P_k + Q_\ell$. \end{lemma}

\begin{proof}
We adapt from the proof given for the Symbolic Power Binomial Theorem \cite[Theorem~3.4]{HNTTBinomial}. Define $L_{t,k,\ell} := \sum_{i+j = t} p_{ik} \cdot q_{j \ell}$ for any integer $ t \ge 1$. 
From the short exact sequences $$0 \rightarrow  L_{t-1,k,\ell}/L_{t,k,\ell}
\rightarrow C/L_{t,k,\ell} \rightarrow C/L_{t-1,k,\ell} \rightarrow 0  \quad \quad  (1 \le t \le n),$$
we may infer that
$$\Ass_C(C/L_{n,k,\ell}) \subseteq \bigcup_{t=1}^n \Ass_C(L_{t-1,k,\ell}/L_{t,k,\ell}).$$
By Display~\eqref{eqn:decomp}, we have 
$$
\Ass_C(L_{t-1,k,\ell}/L_{t,k,\ell})
= \bigcup_{i+j=t-1}\Ass_C\big(p_{i, k} /p_{i + 1, k}  \otimes_K q_{j, \ell}/q_{j+1, \ell}\big).
$$
When the ideal $p_{i, k} /p_{i + 1, k} \subseteq A / p_{i+1, k}$ is non-zero,
its only associated prime ideal is $P_k$,
and similarly
the only associated prime ideal of
$q_{j, \ell}/q_{j+1, \ell}$ is $Q_\ell$.
Thus by Display~\eqref{eqn:SwanWalkKtensor-assoc-primes}  we observe that  
$$\Ass_C(L_{t-1,k,\ell}/L_{t,k,\ell}) \subseteq \Min_C (C/(P_k + Q_\ell)),$$
whence the lemma follows in full. \end{proof}

\begin{lemma}\label{lem:mainpowers}
Let $i$ and $j$ be the least positive integers
such that $P$ is associated to $I^i$ and $Q$ is associated to $J^j$.
Let $\mathcal{P}  \in \Min(C/(P+Q))$.
Then $\mathcal{P}$ is associated to $(I+J)^{i+j-1}$
and to no lower power of $I+J$.
\end{lemma}

\begin{proof}
By Lemma~\ref{lem:colon2},
there exist $f \in I^{i-1}$ and $g \in J^{j-1}$ 
such that $P = I^i : f$
and $J^j : g = Q$.
Then
$$
Pfg \subseteq I^i g \subseteq I^i J^{j-1} \subseteq (I+J)^{i+j-1},
$$
and similarly $Qfg \subseteq (I+J)^{i+j-1}$.
In particular,
$$
P+Q \subseteq (I+J)^{i+j-1} : f g
\subseteq \bigl(I^i + J^j\bigr) : f g,
$$
and by Lemma~\ref{lem:colon},
this is a subset of $(I^i : f) + (J^j : g) = P + Q$.
Thus $P+Q = (I+J)^{i+j-1} : f g$.
Since $\mathcal{P}$ is minimal over $P+Q$,
there exists $c \in C$ such that $(P+Q) : c = \mathcal{P}$,
so that $\mathcal{P} = (I+J)^{i+j-1} : c f g$,
which means that $\mathcal{P}$ is associated to $(I+J)^{i+j-1}$.

Now suppose that $\mathcal{P}$ is associated to $(I+J)^n$.
By Lemma~\ref{lem:main002},
$\mathcal{P}$ is associated to some
$L_{n,k,\ell} = \sum_{m=1}^n p_{mk} q_{n-m,\ell}$. 
By Lemma~\ref{lem:main003},
$\mathcal{P}$ is minimal over $P_k + Q_\ell$. 
By Lemma~\cite[Lemma 2.4]{HNTTBinomial} we conclude that
$P = P_k$ and $Q = Q_\ell$.
Since $p_{0k} = p_{1k} = \cdots = p_{i-1,k} = C = q_{0\ell} = q_{1\ell} = \cdots = q_{j-1,\ell}$,
$$
L_{n,k,\ell} = p_{n-j+1,k} + q_{n-i+1,\ell}
+ \sum_{m=i}^{n-j} p_{mk} q_{n-m,\ell},
$$
and for $P_k$-primary component to appear,
$n-j+1 \ge i$,
i.e., $n \ge i + j - 1$.
\end{proof}

At last, we are now set to deduce Theorems  \ref{thm:SwanWalkC} and \ref{thm:SwanWalkA} and Corollary \ref{thm:SwanWalkB} from the Introduction.

\begin{proof}[Proof of Theorem \ref{thm:SwanWalkC}]
We want to show that for each integer $n \ge 1$, $$
\Ass(C/(I+J)^n)
\subseteq
\bigcup \left\{\mathcal{P} \in \Min(C/(P+Q)) :
P \in \bigcup_{i=1}^n \Ass(C/I^i), Q \in \bigcup_{j=1}^n \Ass(C/J^{j})
\right\},
$$
and that 
$$
\mathcal{A}_C(I+J) =
\bigcup \{\mathcal{P} \in \Min(C/(P+Q)) : P \in
\mathcal{A}(I), Q \in \mathcal{A}(J)\}.
$$ By Lemma~\ref{lem:monotone},
for each index pair $k, l$ with $1 \le k \le r$ and $1 \le \ell \le s$
we can make monotone filtrations $\{p_{i,k}\}_{i \ge 0}$
and $\{q_{j,\ell }\}_{j \ge 0}$ behaving as stipulated in  Definition~\ref{def:powersprimarydecomp}.  
Set $L_{n, k , \ell} = \sum_{i=0}^n p_{i, k} q_{n-i, \ell}$.
By Lemma~\ref{lem:main002},
$(I+J)^n$ is the intersection of the $L_{n,k,\ell}$ as $k$ and $\ell$ vary.
Thus $\Ass(C/(I+J)^n)$
is a subset of $\bigcup_{k , \ell}  \Ass (C/L_{n,k,\ell})$.
Then Lemma~\ref{lem:main003} proves the inclusions $\subseteq$
in the two displays. The opposite inclusion in the latter display follows by Lemma \ref{lem:mainpowers}. 
\end{proof}
\begin{remark}\label{rem:SwanWalkstrictinclusion}
Equality may fail in the first display involving $\Ass(C/(I+J)^n)$ -- 
see Example~\ref{exassnnot}.
\end{remark}

\begin{corollary}\label{thm:SwanWalkD}
Let $K$ be an algebraically closed field,
let $A = K[x_1, \ldots, x_a]$,  
$B = K[y_1, \ldots, y_b]$, 
and $C = A \otimes_K B = K[x_1, \ldots , x_a, y_1, \ldots, y_b]$ be polynomial $K$-algebras.
Let $I$ be an ideal in $A$ and $J$ an ideal in $B$, 
and suppose that for each pair $1 \le k \le r$ and $1 \le \ell \le s$, the collections 
$\{p_{mk}\}_{m=0}^\infty$ and $\{q_{m\ell}\}_{m=0}^\infty$ from Definition \ref{def:powersprimarydecomp} are filtrations  
$($perhaps manufactured via Lemma~\ref{lem:monotone} first$)$.
Then for each positive integer $n$, 
$$
(I+J)^n =
\bigcap_{\ell = 1}^s
\bigcap_{k = 1}^r
\left(
\sum_{i=0}^n p_{ik} \cdot q_{n-i,\ell} 
\right)
$$ is a possibly redundant primary decomposition. Furthermore,
$$
\mathcal{A}_C(I+J) =
\bigcup \{P+Q : P \in \mathcal{A}(I), Q \in \mathcal{A}(J)\}
$$
and so in terms of cardinality of sets, we have the relation 
$$
\#\mathcal{A}_C(I+J) = \#\mathcal{A}_C(I) \cdot \#\mathcal{A}_C(J).
$$
\end{corollary}

\begin{proof}
When $K$ is algebraically closed,
the sum of expansions for a prime ideal in $A$ and a prime ideal in $B$
is a prime ideal in $C$ -- see Milne  \cite[Prop.~4.15]{JSMilneAG}. 
Thus by Lemma~\ref{lem:main003},
$\sum_{i=0}^n p_{ik} \cdot q_{n-i,\ell}$
is $(P_k + Q_\ell)$-primary if it is proper. 
The first display in the statement of the corollary
is simply Lemma~\ref{lem:main002}.
The corollary then follows in full as a consequence of Theorem~\ref{thm:SwanWalkC}. \end{proof}

\begin{proof}[Proof of Corollary \ref{thm:SwanWalkB}] 
Let $\mathcal{P}$ be associated to $(I + J)^{n-1}$.
By Theorem~\ref{thm:SwanWalkC},
$\mathcal{P}$ is minimal over an ideal of the form $P_k + Q_\ell$,
where $P_k$ is associated to $I^i$
and $Q_\ell$ is associated to $J^j$
for some $i,j \in \{1,\ldots, n-1\}$.
Indeed take $i$ and $j$ to be the smallest positive integers such that
$P_k$ is associated to $I^i$
and $Q_\ell$ is associated to $J^j$.
Lemma~\ref{lem:mainpowers} then says that $i + j - 1 \le n-1$,
i.e., that $j \le n-i$.
By a result of Katz and Ratliff~\cite[(1.3) Theorem]{KR},
since powers of $J$ are integrally closed,
$Q_\ell$ is associated to $J^{n-i}$ as well.

By Lemma~\ref{lem:colon2},
there exists $f \in I^{i-1}$ such that $P_k = I^i : f$,
and by Lemma~\ref{lem:colon3},
there exists $g \in J^{n-i-1}$ such that $Q_\ell = J^{n-i} : g$.
Then a proof similar to the beginning part of the proof of Lemma~\ref{lem:mainpowers}
shows that $(I+J)^n : f g = P_k + Q_\ell$,
and so the prime ideal $\mathcal{P}$ minimal over $P_k + Q_\ell$
is associated to $(I+J)^n$. The corollary follows in full. \end{proof}

\section{A Few Concluding Examples}\label{sect:examples}

We close with illustrative examples addressing Remark \ref{rem:SwanWalkstrictinclusion}. We use notation as in Definition~\ref{def:powersprimarydecomp}. Recall that given a prime ideal $\mathfrak{q}$ in a Noetherian ring $R$, its \textbf{$c$-th symbolic power ($c \in \Z_{>0}$)} is $$\mathfrak{q}^{(c)}=   \mathfrak{q}^c R_\mathfrak{q} \cap R = \{f \in R \colon uf \in \mathfrak{q}^c \mbox{ for some }u \in R - \mathfrak{q} \}.$$

\begin{example}
In this example,
for all $n \ge 1$,
$$
\Ass(C/(I+J)^n) = \left\{P + Q:
P \in \bigcup_{i=1}^n \Ass(C/I^i),
Q \in \bigcup_{i=1}^n \Ass(C/J^i)\right\}.
$$
Let $I = (x_1^4, x_1^3 x_2, x_1^2 x_2^2 x_3, x_1 x_2^3, x_2^4)$
and $J = (y_1^3-y_2y_3,y_2^2-y_1y_3,y_3^2-y_1^2y_2)$.
Then $P_1 = (x_1, x_2)$ is a minimal prime over $I$
and $P_2 = (x_1, x_2, x_3)$ is associated only to $I$ and to no other power
of $I$.
For all $n \ge 1$
we have $p_{n1} = (x_1,x_2)^{4n}$,
and for filtration sake we set
$p_{n2} = p_{12} = (x_1^4, x_1^3 x_2, x_1 x_2^3, x_2^4, x_3)$.
The ideal $Q_1 = J$ is the prime ideal defining the monomial curve
$(t^3, t^4, t^5)$,
and by \cite{huneke1986primary},
$Q_2 = (y_1,y_2,y_3)$ is associated to all higher powers of $J$.
Here $q_{n1} = J^{(n)}$,
$q_{12} = C$ and other $q_{n2}$ are proper ideals
such that $\{q_{n2}\}_{n=0}^\infty$ is a filtration.
Below we need the fact that for all positive $n$,
$J^{n-1} \cap q_{n\ell}$ properly contains $J^n$.
For $\ell = 2$ this holds by comparing the $J$-primary components
and noting that the symbolic powers of $J$ are distinct,
and for $\ell = 1$ we first observe that
for a large integer $M$,
$q_{n-1,1} = J^{n-1} : y_1^M$ and
$q_{n1} = J^n : y_1^M$,
so that $q_{n-1,2} = J^{n-1} + (y_1^M)$
and $q_{n2} = J^n + (y_1^M)$.
Then
$$
J^{n-1} \cap q_{n2}
= J^{(n-1)} \cap (J^n + (y_1^M))
= J^n + J^{(n-1)} \cap (y_1^M)
= J^n + y_1^M J^{(n-1)},
$$
and $y_1^M J^{(n-1)}$ is not a subset of $J^n$
as it is not a subset after localizing at $J$.
Thus $J^{n-1} \cap q_{n1}$ properly contains $J^n$.
By Lemma~\ref{lem:main002},
$$
I+J
= \bigcap_{k = 1}^2 \left( p_{0k} q_{11} + p_{1k} q_{01} \right)
= \bigcap_{k = 1}^2 \left( q_{11} + p_{1k} \right)
= \bigcap_{k = 1}^2 \left( J + p_{1k} \right),
$$
and clearly both components are needed. For $n \ge 1$,
$$
(I+J)^n = \bigcap_{\ell = 1}^2 \bigcap_{k = 1}^2
\left( \sum_{i=0}^n p_{ik} q_{n-i,\ell} \right),
$$
and here all four components are needed as we prove next.
By Lemma~\ref{lem:main001},
$$
\bigcap_{\ell = 1}^2 \left( \sum_{i=0}^n p_{i1} q_{n-i,\ell} \right)
= \sum_{i=0}^n p_{i1} J^{n-i}
= \sum_{i=0}^n (x_1,x_2)^{4i} J^{n-i}
= \left((x_1,x_2)^4 + J\right)^n.
$$
The intersection of this with
$\sum_{i=0}^n p_{i2} q_{n-i,\ell}
= q_{n\ell} + p_{12} \sum_{i=1}^n q_{n-i,\ell}
= q_{n\ell} + (x_1^4, x_1^3 x_2, x_1 x_2^3, x_2^4, x_3)$,
contains $x_1^2x_2^2 (J^{n-1} \cap q_{n\ell})$
which is not in $(I+J)^n$.
Thus
for $n \ge 2$
in the intersection of $(I+J)^n$
as in Lemma~\ref{lem:main002}
we cannot omit any component involving $P_2$.
We certainly cannot omit the minimal component $P_1 + Q_1$,
and we cannot omit the component for $P_1 + Q_2$
because
\begin{align*}
\left( \sum_{i=0}^n p_{i1} q_{n-i,1} \right)
&\cap \left( \sum_{i=0}^n p_{i2} q_{n-i,1} \right)
\cap \left( \sum_{i=0}^n p_{i2} q_{n-i,2} \right)
=
\left( \sum_{i=0}^n (x_1,x_2)^{4i} q_{n-i,1} \right)
\cap \left( \sum_{i=0}^n p_{i2} J^{n-i} \right) \\
&=
\left( \sum_{i=0}^n (x_1,x_2)^{4i} q_{n-i,1} \right)
\cap \left( J^n + (x_1^4, x_1^3 x_2, x_1 x_2^3, x_2^4, x_3) \right)
\\
\end{align*}
contains $x_3 q_{n1}$
and is thus not a subset of $(I+J)^n$.
This proves that for all $n \ge 2$,
$(I+J)^n$ has four associated primes.
\end{example}

\begin{example}\label{exassnnot}
In this example,
for all $n \ge 2$,
$$
\Ass(C/(I+J)^n) \subsetneqq \left\{P + Q:
P \in \bigcup_{i=1}^n \Ass(C/I^i),
Q \in \bigcup_{i=1}^n \Ass(C/J^i)\right\}.
$$
Let $I$ be as in the previous example,
and let $J = (y_1^4, y_1^3 y_2, y_1^2 y_2^2 y_3, y_1 y_2^3, y_2^4)$,
which is the ideal $I$ when replacing $x_i \mapsto y_i$.
Thus $I$ and $J$ each have two associated primes
and higher powers have only one associated prime.
It is straightforward to show that $I+J$ has four associated primes,
namely all the combinations $P_i + Q_j$.
We prove next that for all $n \ge 2$,
$P_2 + Q_2$ is not associated to $(I+J)^n$,
i.e.,
that the component $p_{12} + q_{12}$ is redundant in the intersection
in the Lemma~\ref{lem:main001}.
Namely,
\begin{align*}
\left( \sum_{i=0}^n p_{i1} q_{n-i,1} \right)
&\cap \left( \sum_{i=0}^n p_{i1} q_{n-i,2} \right)
\cap \left( \sum_{i=0}^n p_{i2} q_{n-i,1} \right)
= \left( \sum_{i=0}^n p_{i1} J^{n-i} \right) 
\cap \left( q_{n1} + p_{12} \right), \\
\end{align*}
and by the nature of monomial ideals
and since $q_{n1} = J^n$,
this intersection equals
$$
\sum_{i=0}^n p_{i1} \left(J^{n-i} \cap q_{n1} \right) 
+ \sum_{i=0}^n \left( p_{i1} \cap p_{12}\right) J^{n-i}
=
J^n
+ \sum_{i=0}^n I^i J^{n-i}
= (I+J)^n.
$$
\end{example}


\bibliography{SwanWalk} 

\end{document}